\documentclass{article}
\usepackage{amsfonts,graphicx,lscape}

\usepackage{amssymb}
\usepackage{eucal}
\usepackage{amsmath}
\usepackage{amsthm}
\usepackage{amscd}
\usepackage{graphics}
\usepackage{graphicx}
\usepackage{amsfonts}
\usepackage{latexsym}
\usepackage[all]{xy}

\numberwithin{equation}{section}
\newtheorem{thm}{\bf Theorem}[section]

\newtheorem{prop}[thm]{\bf Proposition}

\theoremstyle{remark}
\newtheorem{rem}{\bf Remark}[section]

\newcommand{\bc}{\begin{center}}
\newcommand{\ec}{\end{center}}
\newcommand{\bec}{\begin{equation}}
\newcommand{\eec}{\end{equation}}
\newcommand{\bea}{\begin{eqnarray}}
\newcommand{\eea}{\end{eqnarray}}
\newcommand{\ba}{\begin{array}}
\newcommand{\ea}{\end{array}}

\newcommand{\f}{\displaystyle\frac}

\def\ds{\displaystyle}
\def\vs{\vspace{4pt}}

\begin{document}


\title{Symmetries of the Maxwell-Bloch equations with the rotating wave approximation}
\author{Ioan Ca\c su\\
{\small Departamentul de Matematic\u a, Universitatea de Vest din Timi\c soara}\\
{\small Bd. V. P\^arvan, Nr. 4, 300223 Timi\c soara, Rom\^ania}\\
{\small E-mail: casu@math.uvt.ro}}
\date{}

\maketitle

\begin{abstract}
\noindent In this paper a symplectic realization for the Maxwell-Bloch equations with the rotating wave approximation is given, which also leads to a Lagrangian formulation. We show how Lie point symmetries generate a third constant of motion for the considered dynamical system.
\end{abstract}


\section{Introduction}

The Maxwell-Bloch laser equations have significant importance in optics where they describe the interaction between laser light
and a material sample composed of two-level atoms. A nice presentation of the Maxwell-Bloch dynamics is given by David and Holm in \cite{DavHol92}.
Besides its physical interest, the 3-dimensional real valued Maxwell-Bloch equations
have been widely investigated from the different points of view: homoclinic chaos \cite{HoKoSu91}, Lie-Poisson Hamiltonian structures
\cite{DavHol92}, integrability and geometric prequantization \cite{Puta98}, symmetries \cite{DamPas95}, periodic orbits and energy-Casimir map
\cite{LaBiPa10}. Considering a control, Puta \cite{Puta96} studied the stability problem and B\^ inzar \& L\u azureanu
\cite{LazBin12,BinLaz13} studied some properties of energy-Casimir map.

The present work deals with five dimensional real valued Maxwell-Bloch equations with the rotating wave approximation,
\begin{equation}\label{eq1.1}
\left\{\begin{array}{l}
\dot x_1=y_1\\
\dot y_1=x_1z\\
\dot x_2=y_2\\
\dot y_2=x_2z\\
\dot z=-(x_1y_1+x_2y_2).
\end{array}\right.
\end{equation}

This system has been recently studied by Huang \cite{Huang04} (bi-Hamiltonian structure, homoclinic orbits) and Birtea \& Ca\c su \cite{BirCas13}
(stability of equilibria, homoclinic and periodic orbits).

In the following some symmetries of system (\ref{eq1.1}) are considered. The importance of knowing the symmetry group is reflected by using it to determine some special types of solutions.

The symmetry approach for even order systems of differential equations can be found in \cite{Leach81,Dam90,DamSof99,DamSof00}.

For a class of three dimensional dynamical systems different types of symmetries have been computed in \cite{LazBin12a,BinLaz12,LazBin13}.

Theoretical details about symmetries of differential equations can be found in \cite{Fuc83}, \cite{FokFuc81}, \cite{Olver86}, \cite{BluKum89}, \cite{Dam00}.

The second section of the present paper studies the Hamiltonian structure of the considered system and its Lagrangian representation. Thus, a Lie group is determined and its associated Lie algebra defines a Poisson structure. Moreover, a symplectic realization and a Lagrangian realization of system (\ref{eq1.1}) are given.

In section three, the Lie point symmetries of the Euler-Lagrange equations are studied. These symmetries form a four dimensional Lie algebra and some of them are variational symmetries related with the constants of motion of our system. Also, Lie point symmetries and master symmetries are found.

\section{Hamiltonian structures}

In this section a Poisson structure of system (\ref{eq1.1}) is constructed and a symplectic realization of the system \eqref{eq1.1} is given.

Recall that for system (\ref{eq1.1}), the functions $H,C\in{\cal C}^\infty (\mathbb{R}^5,\mathbb{R})$
$$~H(x_1,y_1,x_2,y_2,z)=\f{1}{2}\left(y_1^2+y_2^2+z^2\right),$$
and
$$C(x_1,y_1,x_2,y_2,z)=\f{1}{2}\left(x_1^2+x_2^2\right)+z$$
are constants of motion.

In order to obtain a Poisson structure, let us consider the linear Poisson bracket $\{\cdot,\cdot\}$,
\begin{equation}\label{2}
\{u_i,u_j\}=\sum\limits_{k=1}^5\alpha_{ij}^ku_k+\beta_{ij}~,~i<j,
\end{equation}
where $u_1=x_1,u_2=y_1,u_3=x_2,u_4=y_2,u_5=z$ and $\alpha_{ij}^k,\beta_{ij}\in\mathbb{R}$.

Imposing the condition that $C$ is a Casimir for \eqref{2} and $H$ is a Hamiltonian function, we get a dynamical system which coincides with the system (\ref{eq1.1}) only if
$$\{u_1,u_2\}=1,~\{u_2,u_5\}=u_1,~\{u_3,u_4\}=1,\{u_4,u_5\}=u_3$$
and $\{u_i,u_j\}=0$ otherwise.

Therefore we consider the five-dimensional Lie algebra given by
$$[E_2,E_5]=E_1,[E_4,E_5]=E_3,$$
where
$$E_1=\left[\begin{array}{cccc}
0&0&-1 &0 \\
0&0&0&0\\
0&0&0&0\\0&0&0&0
\end{array}\right];~~
E_2=\left[\begin{array}{cccc}
0&0&0&0\\
0&0&-1 &0 \\
0&0&0&0\\0&0&0&0
\end{array}\right];~~E_3=\left[\begin{array}{cccc}
0&0&0&1 \\
0&0&0&0\\
0&0&0&0\\0&0&0&0
\end{array}\right];$$
$$
E_4=\left[\begin{array}{cccc}
0&0&0&0\\
0&0&0&1 \\
0&0&0&0\\0&0&0&0
\end{array}\right];~~E_5=\left[\begin{array}{cccc}
0&-1&0&0\\
0&0&0&0 \\
0&0&0&0\\0&0&0&0
\end{array}\right].$$

The Lie algebra generated by the base $B=\{E_1,E_2,E_3,E_4,E_5\}$ is the subalgebra
$$\mathfrak{g}=n_5^5=\{E=\left[\begin{array}{cccc}
0&-\theta &-\alpha &\gamma \\
0&0&-\beta &\delta \\
0&0&0&0\\0&0&0&0
\end{array}\right]|~~\alpha ,\beta ,\gamma ,\delta ,\theta\in\mathbb{R}\}$$
of the Lie algebra $g_4$, see \cite{Ben06}, and the corresponding Lie group is given by
$$N_5^5=\{X=\left[\begin{array}{cccc}
1&-r&-m&p \\
0&1&-n&q\\
0&0&1&0\\0&0&0&1
\end{array}\right]|~~m,n,p,q\in\mathbb{R}\}.$$

We observe that the Lie algebras $(\mathfrak{g},+,\cdot,[\cdot,\cdot])$ and $(\mathbb{R}^5,+,\cdot,\times )$ are isomorphic,
where the product $\times :\mathbb{R}^5\times \mathbb{R}^5\rightarrow\mathbb{R}^5$ is defined by
$$
(\alpha_1,\beta_1,\gamma_1,\delta_1,\theta_1)\times (\alpha_2,\beta_2,\gamma_2,\delta_2,\theta_2)=
(\beta_1\theta_2-\beta_2\theta_1,0,\delta_1\theta_2-\delta_2\theta_1,0,0).
$$
Indeed, an easy computation shows that the map
$$
\Phi :(\alpha ,\beta ,\gamma ,\delta ,\theta )\in\mathbb{R}^5\mapsto
\left[\begin{array}{cccc}
0&-\theta &-\alpha &\gamma \\
0&0&-\beta &\delta \\
0&0&0&0\\0&0&0&0
\end{array}\right]\in\mathfrak{g}
$$
is a Lie algebra isomorphism.

Let us consider the bilinear map $\Theta :\mathfrak{g}\times \mathfrak{g}\to\mathbb{R}$ given by the matrix
$$\Theta =\left[\begin{array}{ccccc}
0&1&0&0&0\\-1&0&0&0&0\\0&0&0&1&0\\0&0&-1&0&0\\0&0&0&0&0\end{array}\right].$$
Following \cite{LibMar87}, $\Theta$ is a 2-cocycle on $\mathfrak{g}$ and it is not a coboundary since
$\Theta (E_1,E_2)=1\not=0=f([E_1, E_2])$, for every linear map $f,~f:\mathfrak{g}\to\mathbb{R}.$\\
Therefore, on the dual space $\mathfrak{g}^*\simeq\mathbb{R}^5$, a modified Lie-Poisson structure
is given in coordinates by
$$\pi =\left[\begin{array}{ccccc}0&0&0&0&0\\0&0&0&0&x_1\\0&0&0&0&0\\0&0&0&0&x_2\\0&-x_1&0&-x_2&0\end{array}\right]+
\left[\begin{array}{ccccc}
0&1&0&0&0\\-1&0&0&0&0\\0&0&0&1&0\\0&0&-1&0&0\\0&0&0&0&0\end{array}\right]=
\left[\begin{array}{ccccc}0&1&0&0&0\\-1&0&0&0&x_1\\0&0&0&1&0\\0&0&-1&0&x_2\\0&-x_1&0&-x_2&0\end{array}\right].$$
Hence $(\mathbb{R}^5,\pi ,X_H)$ is a Hamilton-–Poisson realization of the dynamics (\ref{eq1.1}), where
$X_H=(y_1,x_1z,y_2,x_2z,$ $-x_1y_1-x_2y_2).$

The next theorem states that the system (\ref{eq1.1}) has a symplectic realization.
\begin{thm} The Hamilton-Poisson mechanical system $(\mathbb{R}^5,\pi,X_H)$ has a full symplectic realization
$(T^*\mathbb{R}^3\simeq\mathbb{R}^6,\omega ,X_{\tilde{H}})$, with the canonical symplectic form
$$\omega =\mbox{d}p_1\wedge\mbox{d}q_1+\mbox{d}p_2\wedge\mbox{d}q_2+\mbox{d}p_3\wedge\mbox{d}q_3,$$
the Hamiltonian
$$\tilde{H}=\frac{1}{2}p_1^2+\frac{1}{2}p_2^2+\frac{1}{2}\left[p_3-\f{1}{2}(q_1^2+q_2^2)\right]^2$$
and the corresponding Hamiltonian vector field is given by
$$X_{\tilde{H}}=p_1\frac{\partial }{\partial q_1}+p_2\frac{\partial }{\partial q_2}+
\left[p_3-\f{1}{2}(q_1^2+q_2^2)\right]\frac{\partial }{\partial q_3}+
\left[q_1p_3-\f{1}{2}q_1^3-\f{1}{2}q_1q_2^2\right]\frac{\partial }{\partial p_1}+
\left[q_2p_3-\f{1}{2}q_1^2q_2-\f{1}{2}q_2^3\right]\frac{\partial }{\partial p_2}.$$
\end{thm}
\begin{proof} For the Hamiltonian $\tilde{H}$ the corresponding Hamilton's equations are
\begin{equation}\label{eq1.2}
\left\{\begin{array}{l}
\dot{q}_1=p_1\\
\dot{q}_2=p_2\vs\\
\dot{q}_3=p_3-\f{1}{2}(q_1^2+q_2^2)\vs\\
\dot{p}_1=q_1p_3-\f{1}{2}q_1^3-\f{1}{2}q_1q_2^2\vs\\
\dot{p}_2=q_2p_3-\f{1}{2}q_1^2q_2-\f{1}{2}q_2^3\vs\\
\dot{p}_3=0.
\end{array}\right.
\end{equation}
We define the application $$\varphi :\mathbb{R}^6\to\mathbb{R}^5~,~
\varphi (q_1,q_2,q_3,p_1,p_2,p_3)=(x_1,y_1,x_2,y_2,z),$$
 where
$$x_1=q_1,~y_1=p_1,~x_2=q_2,~y_2=p_2,~z=p_3-\f{1}{2}(q_1^2+q_2^2).$$
It follows that $\varphi $ is a surjective submersion, the Hamiltonian vector field $X_{\tilde{H}}$ is mapped onto the Hamiltonian vector field $X_{H}$ (the equations (\ref{eq1.2}) are mapped onto the equations (\ref{eq1.1})), the
canonical structure $\{.,.\}_{\omega }$ is mapped onto the Poisson structure $\pi $, as required.

We remark that $H\circ\varphi =\tilde{H}$. 
\end{proof}
We also denote $\tilde{C}:=C\circ\varphi =p_3.$

The following result shows that system (\ref{eq1.2}) can be written in Lagrangian formalism.
\begin{thm}
The system (\ref{eq1.2}) takes the form
\begin{equation}\label{eq1.3}
\left\{\begin{array}{l}
\ddot{q}_1-q_1\dot{q}_3=0\\
\ddot{q}_2-q_2\dot{q}_3=0\\
\ddot{q}_3+q_1\dot{q}_1+q_2\dot{q}_2=0
\end{array}\right.
\end{equation}
on the tangent bundle $T\mathbb{R}^3$.\\
Moreover, the system (\ref{eq1.3}) represents the Euler-Lagrange equations generated by the Lagrangian
$$L=\frac{1}{2}\dot{q}_1^2+\frac{1}{2}\dot{q}_2^2+\frac{1}{2}\dot{q}_3^2+\frac{1}{2}\dot{q}_3\left(q_1^2+q_2^2\right).$$
\end{thm}
\begin{proof}
From Hamilton's equations (\ref{eq1.2}) we obtain by differentiation
equations (\ref{eq1.3}). Also, for the Lagrangian $L$, the Euler-Lagrange equations $\displaystyle{\frac{d}{dt}\frac{\partial L}{\partial\dot{q}_i}-\frac{\partial L}{\partial q_i}=0}$ takes
the form (\ref{eq1.3}).

Using Legendre transform $\mathbb{F}L:T\mathbb{R}^3\to T^*\mathbb{R}^3$, $\mathbb{F}L(q_1,q_2,q_3,
\dot{q}_1,\dot{q}_2,\dot{q}_3)=(q_1,q_2,q_3,p_1,p_2,p_3)$, where
$p_i=\displaystyle\frac{\partial L}{\partial\dot{q}_i},$
 the relation between the Hamiltonian $\tilde{H}$ and the Lagrangian $L$, $\tilde{H}=\sum p_i\dot{q}_i-L$, holds.
\end{proof}
For details about Lagrangian and Hamiltonian formalism see, for example, \cite{MarRat99}.

\section{Symmetries}

In this section several types of symmetries are studied.
In the beginning, the Lie point symmetries of system (\ref{eq1.3}) are computed. 
Then these symmetries are transformed in Lie point symmetries, respectively symmetries and master symmetries for system (\ref{eq1.1}).

A vector field $${\bf{u}}=\xi (q_1,q_2,q_3,t)\frac{\partial }{\partial t}+\eta_1(q_1,q_2,q_3,t)\frac{\partial }{\partial q_1}+
\eta_2(q_1,q_2,q_3,t)\frac{\partial }{\partial q_2}+\eta_3(q_1,q_2,q_3,t)\frac{\partial }{\partial q_3}$$
is a Lie-point symmetry for Euler-Lagrange equations (\ref{eq1.3}) if the action of its second prolongation on these equations vanishes, where
$$pr^{(2)}({\bf{u}})={\bf{u}}+\sum (\dot{\eta }_i-\dot{\xi }\dot{q}_i)\frac{\partial }{\partial \dot{q}_i}+\sum
\left(\ddot\eta_i-\ddot\xi \dot q_i-2\dot\xi\ddot q_i\right)\frac{\partial }{\partial\ddot q_i}.$$

Thus the following relations are obtained:
\begin{eqnarray}
&&\ddot{\eta }_1-\ddot{\xi }\dot{q}_1-2\ddot{q}_1\dot{\xi }-\eta_1\dot{q}_3-q_1(\dot{\eta }_3-\dot{\xi }\dot{q}_3)=0\nonumber\\
&&\ddot{\eta }_2-\ddot{\xi }\dot{q}_2-2\ddot{q}_2\dot{\xi }-\eta_2\dot{q}_3-q_2(\dot{\eta }_3-\dot{\xi }\dot{q}_3)=0\nonumber\\
&&\ddot{\eta }_3-\ddot{\xi }\dot{q}_3-2\ddot{q}_3\dot{\xi }+\eta_1\dot{q}_1+\eta_2\dot{q}_2+q_1(\dot{\eta }_1-\dot{\xi }\dot{q}_1)+
q_2(\dot{\eta }_2-\dot{\xi }\dot{q}_2)=0.\nonumber
\end{eqnarray}
The resulting equations obtained by expanding $\dot{\xi },\ddot{\xi },\dot{\eta }_1,\ddot{\eta }_1,\dot{\eta }_2,\ddot{\eta }_2,
\dot{\eta }_3,\ddot{\eta }_3$ and replacing $\ddot{q}_1$, $\ddot{q}_2$ and $\ddot{q}_3$ must be satisfied identically in
$t,$ $q_1,$ $q_2,$ $q_3$, $\dot{q}_1,$ $\dot{q}_2$, $\dot{q}_3$, which are all independent variables. By performing straightforward computations, we get
the overall result:
$$\xi =-\alpha t+\beta ~,~~\eta_1=\alpha q_1+\gamma q_2~,~~\eta_2=-\gamma q_1+\alpha q_2~,~~
\eta_3=\alpha q_3+\delta ,$$
where $\alpha ,\beta ,\gamma ,\delta $ are real constants.

We can summarize the above considerations in the following result.
\begin{thm}
The symmetries of equations (\ref{eq1.3}) are given by
\begin{equation}\label{eq1.4}
{\bf u}=(-\alpha t+\beta )\frac{\partial }{\partial t}+(\alpha q_1+\gamma q_2)\frac{\partial }{\partial q_1}+
(-\gamma q_1+\alpha q_2)\frac{\partial }{\partial q_2}+(\alpha q_3+\delta )\frac{\partial }{\partial q_3},
\end{equation}
where $\alpha,\beta,\gamma,\delta\in\mathbb{R}.$
\end{thm}

The next proposition provides the algebraic structure of the above symmetries.
\begin{prop} The symmetries of equations (\ref{eq1.3}) form a 4-dimensional Lie algebra $\mathfrak{s}$. \end{prop}
\begin{proof} This Lie algebra is generated by the base $\{{\bf u}_1,{\bf u}_2,{\bf u}_3,{\bf u}_4\}$, where
$$\begin{array}{l}
{\bf u}_1=\ds{-t\frac{\partial }{\partial t}+q_1\frac{\partial }{\partial q_1}+q_2\frac{\partial }{\partial q_2}
+q_3\frac{\partial }{\partial q_3}}\vs\\
\ds{{\bf u}_2=\frac{\partial }{\partial t}}\vs\\
\ds{{\bf u}_3=\frac{\partial }{\partial q_3}}\vs\\
\ds{{\bf u}_4=q_2\frac{\partial }{\partial q_1}-q_1\frac{\partial }{\partial q_2}}~,
\end{array}$$
and the Lie algebra bracket is given by:
$$[{\bf u}_1,{\bf u}_2]={\bf u}_2~,~~[{\bf u}_1,{\bf u}_3]=-{\bf u}_3~,~~
[{\bf u}_1,{\bf u}_4]=0~,~~[{\bf u}_2,{\bf u}_3]=0~,~~
[{\bf u}_2,{\bf u}_4]=0~,~~
[{\bf u}_3,{\bf u}_4]=0.$$
\end{proof}
Now, we consider the matrix Lie algebra generated by the base $B=\{A_1,A_2,A_3,A_4\}$,
$$A_1=\left[\begin{array}{cccc}0&0&0&0\\0&0&0&0\\0&0&-1&0\\0&0&0&1\end{array}\right];~~
A_2=\left[\begin{array}{cccc}0&0&1&0\\0&0&0&0\\0&0&0&0\\0&0&0&0\end{array}\right];$$
$$A_3=\left[\begin{array}{cccc}0&0&0&1\\0&0&0&0\\0&0&0&0\\0&0&0&0\end{array}\right];~~
A_4=\left[\begin{array}{cccc}0&1&0&0\\0&0&0&0\\0&0&0&0\\0&0&0&0\end{array}\right],$$
namely
$$\mathfrak{s_g}=\{A=\left[\begin{array}{cccc}
0 &d&b&c \\
0&0&0&0\\
0&0&-a&0\\0&0&0&a
\end{array}\right]|~~a,b,c,d \in\mathbb{R}\}.$$

The following relations $$~[A_1,A_2]=A_2,~[A_1,A_3]=-A_3,~[A_1,A_4]=0,~[A_2,A_3]=0,~[A_2,A_4]=0,~[A_3,A_4]=0$$ hold. Hence the Lie algebras
$\mathfrak{s}$ and $\mathfrak{s_g}$ are isomorphic.

The Lie group corresponding to Lie algebra $\mathfrak{s_g}$ is given by
$$S_G=\{X=\left[\begin{array}{cccc}
1&w&u&v \\
0&1&0&0\\
0&0&e^{-s}&0\\0&0&0&e^s
\end{array}\right]|~~s,u,v,w \in\mathbb{R}\}.$$

The following proposition furnishes variational symmetries of Euler-Lagrange equations (\ref{eq1.3}).
\begin{prop}
In the case $\alpha =0$ the Lie point symmetries ${\bf u}$ given by (\ref{eq1.4}) are variational symmetries of equations (\ref{eq1.3}).
\end{prop}
\begin{proof}
The vector field ${\bf u}$ with the infinitesimal generators $\xi,\eta_1,\eta_2,\eta_3$ is a variational symmetry if and only if $~pr^{(1)}({\bf u})L+L\dot{\xi }=0$, see \cite{BB04}. In our case
\begin{eqnarray}
pr^{(1)}({\bf u})&=&(-\alpha t+\beta )\frac{\partial }{\partial t}+(\alpha q_1+\gamma q_2)\frac{\partial }{\partial q_1}+
(-\gamma q_1+\alpha q_2)\frac{\partial }{\partial q_2}+(\alpha q_3+\delta )\frac{\partial }{\partial q_3}\nonumber\\
&+&(2\alpha\dot{q}_1+\gamma \dot{q}_2)\frac{\partial }{\partial\dot{q}_1}+(2\alpha\dot{q}_2-\gamma \dot{q}_1)\frac{\partial }{\partial\dot{q}_2}+2\alpha\dot{q}_3\frac{\partial }{\partial\dot{q}_3}.\nonumber
\end{eqnarray}
Therefore  $~pr^{(1)}({\bf u})L+L\dot{\xi }=3\alpha L$ and the conclusion follows.
\end{proof}

\begin{rem}
It is known that variational symmetries give rise to constants of motion. More precisely, using Noether's theorem (\cite{BB04,Noether18}) we obtain that
$$I=-\beta\tilde{H}-\gamma\tilde{J}+\delta\tilde{C}$$
is constant of motion for equations (\ref{eq1.3}), where $\tilde{J}=q_1\dot{q}_2-q_2\dot{q}_1$, or, using Legendre transformation, 
$\tilde{J}=q_1p_2-q_2p_1.$
Moreover,
${\bf u}_2$ represents the time translation symmetry which generates the conservation of energy $\tilde{H}$,
 ${\bf u}_3$ represents a translation in the cyclic $q_3$ direction which is related to the conservation of the conjugate momentum  
 $p_3=\tilde{C}$. Also, the variational symmetry ${\bf u}_4$ represents a rotation around $q_3$ axis and the corresponding
constant of motion $\tilde{J}=q_1p_2-q_2p_1$ is the third component of the angular momentum vector ${\bf q}\times {\bf p}$.
\end{rem}
\begin{thm}
Solving $\tilde{J}=J\circ\varphi $ it follows that $J=x_1y_2-x_2y_1$ is a third constant of motion for the system (\ref{eq1.1}).
\end{thm}

Taking into account the relationship between Maxwell-Bloch equation (\ref{eq1.1}) and Euler-Lagrange equations (\ref{eq1.3}) 
it is natural to ask what connections are between the symmetries of these systems.

Using the push forward on a vector field by $\mathbb{F}L$, one gets the corresponding vector field on $T^*\mathbb{R}^3$:
\begin{eqnarray}
(\mathbb{F}L)_*(pr^{(1)}({\bf u}))&=&(-\alpha t+\beta )\frac{\partial }{\partial t}+(\alpha q_1+\gamma q_2)\frac{\partial }{\partial q_1}+
(-\gamma q_1+\alpha q_2)\frac{\partial }{\partial q_2}+(\alpha q_3+\delta )\frac{\partial }{\partial q_3}\nonumber\\
&+&(2\alpha p_1+\gamma p_2)\frac{\partial }{\partial p_1}+(2\alpha p_2-\gamma p_1)\frac{\partial }{\partial p_2}+2\alpha p_3\frac{\partial }{\partial p_3},\nonumber
\end{eqnarray}
denoted by $\tilde{\bf v}$. Applying the push forward on the vector field $\tilde{\bf v}$ by $\varphi $ one obtains the following vector field:
\begin{equation}\label{X}
\begin{array}{r@{}l}
{\bf X}&{}=(-\alpha t+\beta )\frac{\partial }{\partial t}+(\alpha x_1+\gamma x_2)\frac{\partial }{\partial x_1}+
(2\alpha y_1+\gamma y_2)\frac{\partial }{\partial y_1}+
(-\gamma x_1+\alpha x_2)\frac{\partial }{\partial x_2}\\
&{}\\
&{}+(2\alpha y_2-\gamma y_1)\frac{\partial }{\partial y_2}+2\alpha z\frac{\partial }{\partial z}.
\end{array}
\end{equation}

Now, we can present symmetries of Maxwell-Bloch equations (\ref{eq1.1}).
\begin{prop}
\indent $(i)$ The vector field 
\eqref{X}
is a Lie point symmetry of
(\ref{eq1.1}). Moreover, it is a conformal symmetry and also a master symmetry.\\
$(ii)$ In the case $\alpha =0$, the vector field
$${\bf X}=\beta \frac{\partial }{\partial t}+\gamma x_2\frac{\partial }{\partial x_1}+\gamma y_2\frac{\partial }{\partial y_1}-
\gamma x_1\frac{\partial }{\partial x_2}-\gamma y_1\frac{\partial }{\partial y_2},$$
 is a symmetry of (\ref{eq1.1}).
\end{prop}


\bigskip

\noindent {\bf Acknowledgements.} This work was supported by a grant of the Romanian National Authority for Scientific
Research, CNCS - UEFISCDI, project number PN-II-RU-TE-2011-3-0006.

\end{document}